\documentclass[12pt]{amsart}

\title{Counting isotropic tangent lines of hypersurfaces}
\author{Sergei Lanzat}
\address{Max-Planck-Instituts f\"{u}r Mathematik, 53111
Bonn, Germany}
\email{serjl@mpim-bonn.mpg.de}
\date{\today}
\makeatletter

\usepackage{amsmath,amssymb,amsthm,amsfonts, mathrsfs,amscd}
\usepackage{dsfont}
\usepackage{latexsym}
\usepackage[pdftex]{graphicx}
\input xy
\xyoption{all}
\usepackage{bbm}

\renewcommand{\(}{\left(}
\renewcommand{\)}{\right)}

\newcommand{\FAT}[1]{\mbox{{$\mathbb{#1}$}}}
\newcommand{\Fat}[1]{\mbox{{$\scriptstyle\mathbb{#1}$}}}

\newcommand{\ZZ}{\FAT{Z}}

\newcommand{\NN}{\FAT{N}}

\newcommand{\RR}{\FAT{R}}

\newcommand{\CC}{\FAT{C}}
\newcommand{\SSS}{\FAT{S}}

\newcommand{\rr}{\Fat{R}}

\newcommand{\sss}{\Fat{S}}

\newcommand{\minus}{\smallsetminus}
\newcommand{\eps}{\varepsilon}

\newcommand{\Id}{\mathds{1}}

\newcommand{\be}{\begin{itemize}}
\newcommand{\ee}{\end{itemize}}

\newcommand{\beq}{\begin{equation}}
\newcommand{\eeq}{\end{equation}}

\newcommand{\beqn}{\begin{equation}\nonumber}

\newcommand{\bea}{\begin{equation}\begin{aligned}}
\newcommand{\eea}{\end{aligned}\end{equation}}

\newcommand{\bean}{\begin{equation}\nonumber\begin{aligned}}

\DeclareMathOperator{\End}{End}
\DeclareMathOperator{\Span}{Span}

\DeclareMathOperator{\Crit}{Crit}
\DeclareMathOperator{\ind}{ind}

\DeclareMathOperator{\sgn}{sgn}
\DeclareMathOperator{\Imm}{Imm}

\newtheorem{thm}{Theorem}[section]
\newtheorem{thm*}{Theorem}
\newtheorem{lem}[thm]{Lemma}
\newtheorem{lem*}[thm*]{Lemma}
\newtheorem{prop}[thm]{Proposition}
\newtheorem{cor}[thm]{Corollary}

\newtheorem{defn}[thm]{Definition}
\newtheorem{thm-defn}[thm]{Theorem-Definition}

\theoremstyle{definition}

\begin{document}

\begin{abstract}
Consider the standard symplectic $(\RR^{2n}, \omega_0)$, a
point $p\in\RR^{2n}$ and an immersed closed 
orientable hypersurface $\Sigma\subset\RR^{2n}\minus\{p\}$, all
in general position. We study the following passage/tangency question: 
how many lines in $\RR^{2n}$ pass through $p$ and tangent to $\Sigma$
parallel to the 1-dimensional characteristic distribution 
$\ker\left(\omega_0\big|_{T\Sigma}\right)\subset T\Sigma$ of $\omega_0$.
We count each such line with a certain sign, and present an explicit
formula for their algebraic number. This number is invariant under  regular
homotopies in the class of a general position of the pair $(p, \Sigma)$, but
jumps (in a well-controlled way) when during a homotopy we pass a certain
singular discriminant. It provides a low bound to the actual number of these
isotropic lines.

\end{abstract}


\maketitle

\section{Introduction}
\subsection{Immersions and their invariants.}\label{subsection: immersions}
The study of the topology of the space of immersions  $\Imm(X,Y)$ of a smooth manifold
$X$ into a smooth manifold $Y$ is a famous classical problem. Recall that  a smooth
map $f: X\to Y$ is called an immersion  if its differential $df$ is everywhere injective. 
The space $\Imm(X,Y)$ equipped with the $C^\infty$-topology is an open subset of 
the Fr\'{e}chet manifold $C^\infty(X, Y)$. 

The study of $\Imm(X,Y)$ becomes especially 
interesting  if $X$ and $Y$ are orientable and oriented manifolds.  For example, one can study 
the path-connectedness of $\Imm(X,Y)$. Two immersions are regularly homotopic if they can 
be connected by a continuous path of immersions, i.e.  they belong to the same path-component of $\Imm(X,Y)$. 
Already in the simplest cases one gets non-trivial results. In particular, Whitney \cite{Whitney} classified 
immersions of the oriented unit circle $\SSS^1$ into the oriented Euclidean plane $\RR^2$, 
i.e. oriented plane immersed curves 
$\Gamma: \SSS^1\to\RR^2$ up to regular homotopy:  path-components of $\Imm(\SSS^1, \RR^2)$ are in  a natural bijection with
integers $\ZZ$. The bijection is given by the Whitney index $\ind(\Gamma)$ - the degree of the tangential 
Gauss map  $G_T:\SSS^1\to\SSS^1$ given by  $G_T(s)=\cfrac{\Gamma^\prime(s)}{\|\Gamma^\prime(s)\|}$. Note that 
the Whitney index  $\ind: \Imm(\SSS^1, \RR^2)\to\ZZ$ is a locally constant function, i.e. it is an invariant 
of immersions up to regular homotopies. The Whitney index $\ind(\Gamma)$ can also be expressed in terms of the 
normal Gauss map as follows. Fix the standard orientations: $o_{\sss^{1}}$ on $\SSS^1$ and $o_{\rr^{2}}$ on $\RR^2$. 
Let us coorient $\Gamma(\SSS^1)$ by choosing a normal vector field $N$ on it satisfying 
$N\times o_{\Gamma(\sss^1)} = o_{\rr^{2}}$. Then  $\ind(\Gamma)$ equals to the degree 
of the normal Gauss map $G_N:\SSS^1\to\SSS^1$ given by $G_N(s)=N(\Gamma(s))$. 

The following Morse theoretical  interpretation of $\ind(\Gamma)$ provides an interesting connection between the Morse theory 
and the degree theory. Let  $h: \Gamma(\SSS^1)\to\RR$ 
be a Morse function, such that its set of critical points $\Crit(h)$ does not contain singular points of the curve $\Gamma$. 
For example, one can take a generic height function  $h(p):=\langle p - p_0, v\rangle$, where $p_0\in\RR^2$ and $v\in\SSS^1$, 
see Figure \ref{fig:index}.

\begin{figure}[htb]
\centering
    \includegraphics[width=2.5in]{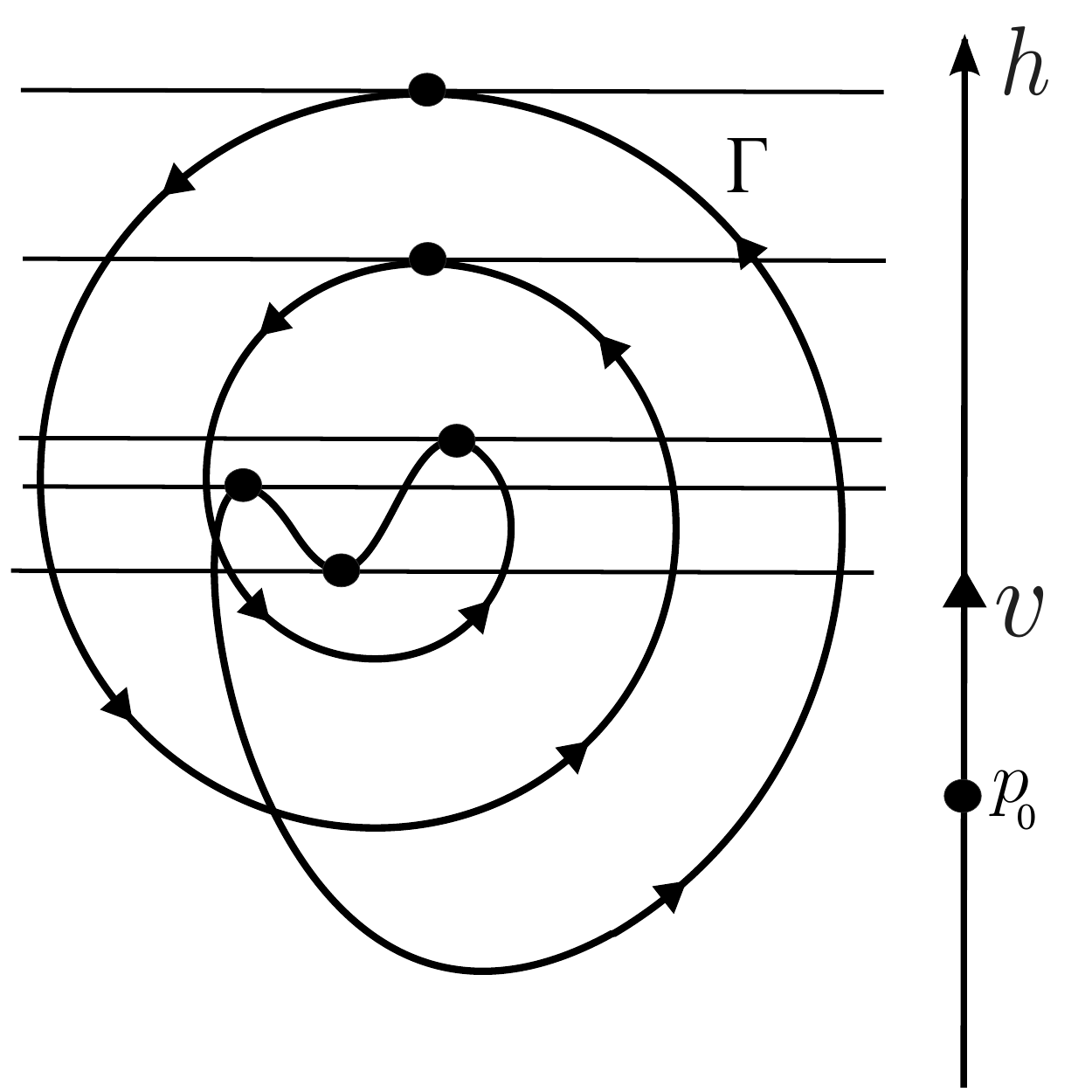}
\caption{Whitney index via Morse function.\label{fig:index}}
\end{figure}

Then $\Crit(h)$ splits into $\Crit^+(h)\sqcup\Crit^-(h)$, where 
$$\Crit^+(h):=\{p\in\Crit(h)| N(p) = + v\}$$ 
and  
$$\Crit^-(h):=\{p\in\Crit(h)| N(p) = - v\}.$$
In particular, we can use  "half" of critical points of $h$ to compute the Whitney index. 
Namely, if we denote by $\mu_h(p)$ the Morse index of a critical point $p$ of the function $h$, then  
$$\ind(\Gamma)=\sum\limits_{p\in\Crit^+(h)}(-1)^{\dim(\sss^1)-\mu_h(p)}=\sum\limits_{p\in\Crit^-(h)}(-1)^{\mu_h(p)}.$$
Indeed, $\ind(\Gamma)=\deg(G_N)$ and since $\pm v$ are regular values of $G_N$ we have
$$\deg(G_N) = \sum\limits_{G_N(p)=+v}\deg_pG_N = \sum\limits_{G_N(p)=-v}\deg_pG_N.$$
Finally, we observe that  local degrees satisfy  $\deg_pG_N=(-1)^{\dim(\sss^1)-\mu_h(p)}$, if $p\in\Crit^+(h)$ 
and  $\deg_pG_N=(-1)^{\mu_h(p)}$, if $p\in\Crit^-(h)$. It is worth pointing out that 
$$\sum\limits_{p\in\Crit^+(h)}(-1)^{\dim(\sss^1)-\mu_h(p)}=-\sum\limits_{p\in\Crit^+(h)}(-1)^{\mu_h(p)}$$
and thus,
$$\chi(\SSS^1)=\sum\limits_{p\in\Crit}(-1)^{\mu_h(p)}=\sum\limits_{p\in\Crit^+(h)}(-1)^{\mu_h(p)}+\sum\limits_{p\in\Crit^-(h)}(-1)^{\mu_h(p)}=0.$$
Note also that while the Whitney index $\ind(\Gamma)$ was defined as a discrete sum of signs (signed points), it can as well be expressed 
as a continuous sum. Namely, by Hopf's Umlaufsatz \cite{Hopf}, we have the integral formula
$$ \int_{\sss^1}G_N^*\mu = \ind(\Gamma),$$ where  $\mu\in\Omega^1(\SSS^1;\RR)$ is the volume form normalized by  $\displaystyle \int_{\sss^1}\mu=1$.

The above description immediately  generalizes to  higher dimensions. Indeed, consider the space $\Imm(S, \RR^m)$, 
where $S$ is a smooth closed orientable manifold of dimension $m-1$.  Fix orientations: the standard $o_{\rr^{m}}$ on $\RR^m$ and $o_S$ on $S$.  Let $\imath: S\to\RR^m$ be an immersion. Then $\Sigma:=\imath(S)\subset\RR^m$ is an immersed orientable hypersurface oriented by $o_{\Sigma}:=\imath_*(o_S)$. 
Let us coorient $\Sigma$ by choosing a normal vector field $N$ on it satisfying $N\times o_{\Sigma} = o_{\rr^m}$. 
Define the Whitney index $\ind(\Sigma)$ of $\Sigma$ to be the degree of the normal Gauss map
$G_N:S\to\mathbb{S}^{m-1}$, $G_N(s) = N(\imath(s))$. By choosing  a  Morse height function $h$
on $\Sigma$ as before, we get  that
$$\ind(\Sigma)=\sum\limits_{p\in\Crit^+(h)}(-1)^{\dim(S)-\mu_h(p)}=\sum\limits_{p\in\Crit^-(h)}(-1)^{\mu_h(p)}.$$
Now, we observe the principal difference between even and odd dimensional cases. Indeed, if  $\dim(S)$ is even, then  
$$\sum\limits_{p\in\Crit^{\uparrow}(h)}(-1)^{\dim(S)-\mu_h(p)}=\sum\limits_{p\in\Crit^{\uparrow}(h)}(-1)^{\mu_h(p)}$$
and so 
$$2\ind(\Sigma)=\sum\limits_{p\in\Crit(h)}(-1)^{\mu_h(p)}=\chi(S),$$
where $\chi(S)$ is the Euler characteristics of $S$. It means that $\ind$ is a constant function $\frac12\chi(S)$ on the space
$\Imm(S, \RR^m)$. In the odd dimensional case, $\ind$ is already a locally constant function, which depends on topological types of immersions. Note also that in the latter case we have $\chi(S)=0$. 
So we have a constant function $\chi(S)$ on the space $\Imm(S, \RR^m)$, which, in the case of $\dim(S)$ is odd, naturally splits into an invariant of immersions up to regular homotopies.

The above description of  $\ind(\Sigma)$ has the following enumerative meaning: 
$2\ind(\Sigma)$  it is an algebraic (with the above Morse signs) 
number of affine tangent spaces to $\Sigma$ parallel to a given one. 
In turn, this can be seen as an algebraic count of affine tangent spaces to $\Sigma$ that pass through $(m-2)$-dimensional space at "infinity". So one can generalize this very special case to a (algebraic) count of affine tangent spaces to $\Sigma$ that pass 
through a generic $(m-2)$-dimensional affine space $P\subset\RR^m\minus\Sigma$. In Morse theoretical terms it means that we consider 
now a Morse function $h$, which is a projection of  $\Sigma$ from $P$ onto a generic directed line, see Figure \ref{fig:coneMorse}.  
\begin{figure}[htb]
\centering
    \includegraphics[width=2.8in]{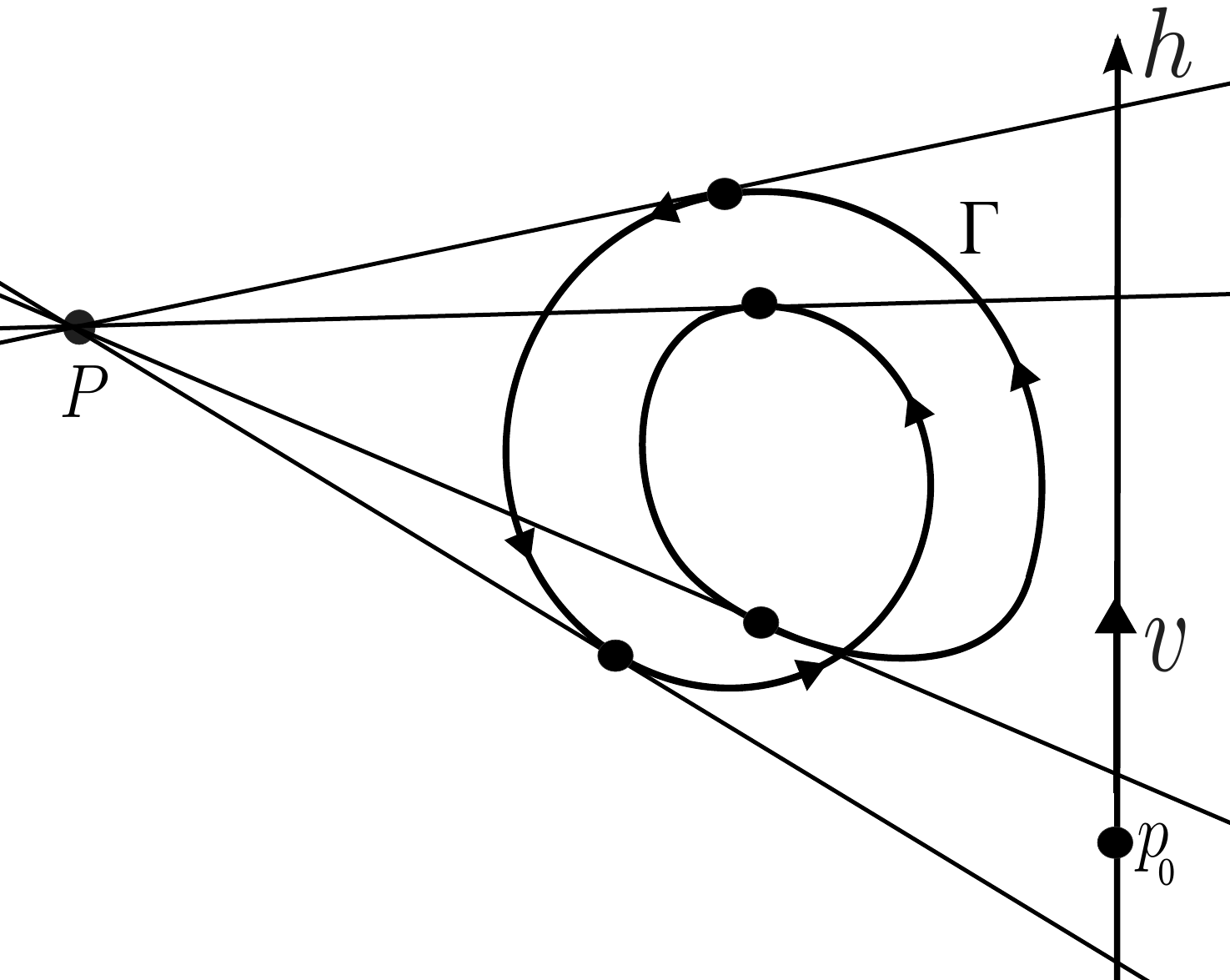}
\caption{Tangent lines through a point.\label{fig:coneMorse}}
\end{figure}
The plane case of $m=2$ was considered in \cite{L-P1}. Namely, let $\mathcal{L}:=\mathcal{L}(p,\Gamma)$ be a set
of lines in $\RR^2$ passing through a fixed point $p$ and tangent to a (generic) oriented immersed
plane  closed curve $\Gamma$. Each such line $l\in \mathcal{L}$ was counted with a certain sign $\eps_l$, so that the total algebraic number $\mathcal{N}(p, \Gamma)=\sum_{l\in\mathcal{L}}\eps_l$ of lines does not change under homotopy of $\Gamma$ in
$\RR^2\minus p$.  One can guess such a sign rule as follows. 
\begin{figure}[htb]
\centering
    \includegraphics[width=5in]{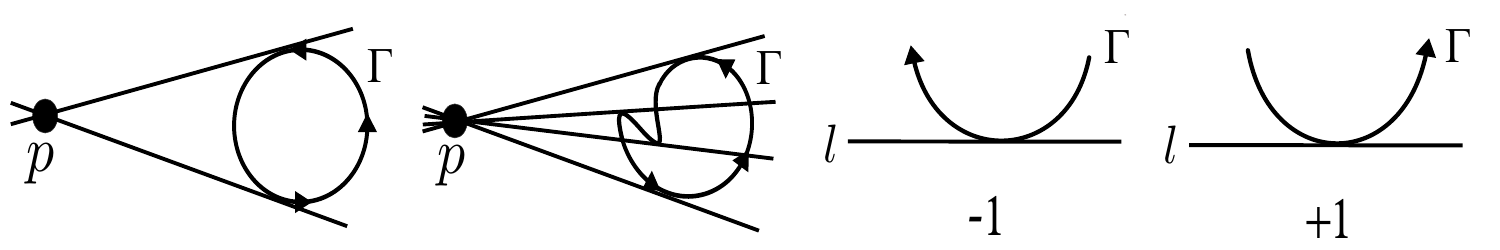}
\centerline{\hspace{0.1in}a\hspace{2.4in}b}
\caption{Counting lines with signs.\label{fig:morse}}
\end{figure}
Under a deformation shown in Figure \ref{fig:morse}a, two new lines
appear, so their contributions to $\mathcal{N}(p, \Gamma)$ should cancel out.
Thus, their signs should be opposite and one gets the sign rule shown in Figure
\ref{fig:morse}b. 
It follows (see \cite{L-P1}) that 
\begin{equation}\label{eq:N motivation}
\mathcal{N}(p, \Gamma)=2\ind(\Gamma)-2\ind_p(\Gamma)\ .
\end{equation}
\noindent Here, $\ind_p(\Gamma)$ is the index of $p$ w.r.t. $\Gamma$, i.e. the number
of turns made by the vector connecting $p$ to a point $q\in\Gamma$,
as $q$ passes once along $\Gamma$ following the orientation. It may
be computed as the intersection number $I([p,\infty],\Gamma;\RR^2)$
of a $1$-chain $[p,\infty]$ (i.e. an interval connecting $p$ with a point near
infinity of $\RR^2$) with an oriented $1$-cycle $\Gamma$ in $\RR^2$.
The appearance of $\ind(\Gamma)$ and $\ind_p(\Gamma)$ in the above
formula comes as no surprise: in fact, these are the only invariants
of the curve $\Gamma$ under its homotopy in the class of immersions
in $\RR\minus p$.

In contrast to the plane case, in higher dimensions we do not have a natural choice of
a tangent direction on an immersed hypersurface  $\Sigma:=\imath(S)\subset\RR^m$, 
$\imath\in\Imm(S, \RR^m)$,  such that the degree of the corresponding  tangent Gauss 
map is a locally constant function on $\Imm(S, \RR^m)$. Nevertheless, such a tangent 
direction can be chosen in the presence of additional structures on the target space $\RR^m$. 
Indeed, let us consider the following equivalent formulation of the plane case. 
Equip $\RR^2$ with the standard symplectic structure $\omega_0$ and ask how many  
lines in $\RR^2$ pass through $p$ and tangent to
$\Gamma$ parallel to 1-dimensional characteristic distribution 
$\ker\left(\omega_0\big|_{T\Gamma}\right)$ of $\omega_0$. 
Indeed, in the 2-dimensional case $\Gamma$ is a Lagrangian immersed 
submanifold, i.e. $\ker\left(\omega_0\big|_{T\Gamma}\right)=T\Gamma$, and the
new formulation is equivalent to the previous one. So in the present paper we
study the following question. Consider the standard  symplectic 
$(\RR^{2n}, \omega_0),\; n\in\NN$ with the standard orientation, a point $p\in\RR^{2n}$ and an
immersed closed  oriented hypersurface $\Sigma\subset\RR^{2n}\minus\{p\}$, all
in general position. How many lines in $\RR^{2n}$ pass through $p$ and tangent
to $\Sigma$ parallel to the 1-dimensional characteristic distribution 
$\ker\left(\omega_0\big|_{T\Sigma}\right)\subset T\Sigma$ of $\omega_0$? 
Such lines will be called isotropic lines of $\Sigma$ passing through $p$.

\subsection{Main results and the structure of the paper.}
Let $p\in\RR^{2n}$ and $\Sigma\subset\RR^{2n}\minus\{p\}$ be as before. 
We use the natural almost contact structure and the shape operator of $\Sigma$
in order to equip each isotropic line with a certain sign of tangency. We count each such
line with this sign and present an explicit formula for their algebraic number.
This number is invariant under regular homotopies in the class of a general position 
of the pair $(p, \Sigma)$, but jumps (in a
well-controlled way) when during a homotopy we pass a certain singular
discriminant.  It provides a low bound to the actual number of these
isotropic lines.

The paper is organized in the following way. In Section 2 we
introduce objects of our study, define signs of tangency, list the
requirements of a general position, and formulate the main theorem.
Section 3 is dedicated to the proofs. We interpret the desired
number of lines as a certain intersection number; the main claim
follows from different ways of its calculation. We also obtain an integral formula
for that number of lines. 

\subsection*{Acknowledgement.}
I am grateful to Michael Polyak  for the constant encouragement, many fruitful
discussions and valuable suggestions. This work was carried out at Max-Planck-Institut
f\"{u}r Mathematik, Bonn, and I would like to acknowledge its excellent research atmosphere 
and hospitality.

\section{Statement of the main results.}
\subsection{Setting. \label{subsect:setting}}
Consider the standard symplectic $(\RR^{2n},\omega_0)$ with
the fixed standard orientation $o_{\rr^{2n}}$. Let
$p\in(\RR^{2n},\omega_0)$ be a fixed point, let $J_0$ be the
standard almost complex structure on $(\RR^{2n},\omega_0)= (\CC^n,
\omega_0)$, i.e. the multiplication by $\sqrt{-1}$. Suppose
that $S$ is a smooth closed orientable manifold of dimension
$2n-1$ with a fixed orientation $o_S$ and $\imath: S
\looparrowright\RR^{2n}\minus\{p\}$ is an immersion. Then
$\Sigma:=\imath(S)\subset\RR^{2n}\minus\{p\}$ is an immersed
orientable hypersurface oriented by
$o_{\Sigma}:=\imath_*(o_S)$. Let us coorient $\Sigma$ by
choosing a normal vector field $N$ on it satisfying 
$N\times o_{\Sigma} = o_{\rr^{2n}}$. The tangent vector
field $J_0N$ on $\Sigma$ spans the kernel of
$\omega_0|_{T\Sigma}$, i.e. spans isotropic lines of
$\omega_0|_{\Sigma}$. Let $(P, -J_0N, \eta)$ be the 
natural almost contact structure on $\Sigma$ induced 
by $J_0$. Namely,  for  $X\in\Gamma(T\Sigma)$ the $(1,1)$-tensor
$P\in\End(\Gamma(T\Sigma))$ and the 1-form
$\eta\in\Omega^1(\Sigma)$ are given by 
$P(X)=J_0(X)-\langle J_0(X), N\rangle N$ and 
$\eta(X)=\langle J_0(X), N\rangle$, where 
$\langle \cdot, \cdot\rangle$ is the standard 
Euclidean inner product on $\RR^{2n}$ -- see \cite{Blair} for more details. 

Finally, denote by $\mathcal{L}:=\mathcal{L}(p, \Sigma, \omega_0)$  the set
of lines in $\RR^{2n}$ passing through a fixed point $p$ and tangent to 
$\Sigma$ parallel to the 1-dimensional characteristic distribution 
$\ker\left(\omega_0\big|_{T\Sigma}\right)\subset T\Sigma$ of $\omega_0$.

\subsection{General position for the pair $(p,\Sigma)$ and signs of lines.
\label{subsect:generalpositionand signs}}
We shall assume that the following (generic) conditions hold:
\begin{itemize}
\item[1.] The hypersurface $\Sigma$ is generically immersed in
$\RR^{2n}\minus\{p\}$, 
i.e. all its self-intersections are transversal.
\item[2.] Every $\ell\in\mathcal{L}$ is tangent to $\Sigma$ at only one
non-singular point.
\item[3.] If a line $\ell\in\mathcal{L}$ is tangent to the hypersurface $\Sigma$ at a point $\imath(s)$,
then $\det(A_s\pm\lambda_sP_s)\neq 0$, where  $A_s$ 
is the shape operator of the immersion 
$\imath:S\looparrowright\RR^{2n}\minus\{p\}$ at $s$, 
$\lambda_s :=\|\imath(s)-p\|^{-1}$ and $P_s$ is a short notation for 
$(d_s\imath)^{-1}\circ P_{\imath(s)}\circ d_s\imath$. 
Recall, that the shape operator is equal to the differential  $d_sG$ of  the Gauss map 
$G:S\to\mathbb{S}^{2n-1}$, $G(s) = N(\imath(s))$ at the point $s$. 
\end{itemize}
Now, to each $\ell\in\mathcal{L}$ we assign a sign
$\varepsilon_\ell\in\{\pm 1\}$ as follows. 
\begin{defn}
Suppose that a line $\ell$ is tangent to $\Sigma$ at a point $\imath(s)$, 
such that its direction vector $\xi_{\imath(s)}:=\cfrac{\imath(s)-p}{\|\imath(s)-p\|}$ equals to 
$\pm J_0N(\imath(s))$, then define
\begin{equation}
\varepsilon_\ell := \sgn(\det(A_s\pm\lambda_sP_s)).
\end{equation}
\end{defn}

\noindent Note that the sign  $\varepsilon_\ell$ is invariant under a
homothety of $\RR^{2n}\minus\{p\}$ with a positive ratio and
the center at $p$. Indeed, under such a homothety
with a positive ratio $c$, both operators $A_s$ and $\lambda_sP_s$
are multiplied by $c^{-1}$. In addition, for $n=1$, i.e. the case of the toy model 
from Section~\ref{subsection: immersions}, we have that $P=0$. 
In particular, the sign $\varepsilon_\ell$ equals to the sign 
of the curvature of a curve at the point of tangency. 
As a consequence, up to an orientation of the curve it coincides
with Polyak's sign. In the higher dimensional case ($n>1$), 
when $\lambda_s <<1$, i.e. the point $p$ is far away from $\Sigma$,
we have  $$\varepsilon_\ell = \sgn(\det(A_s\pm\lambda_sP_s))=\sgn(\det(A_s))=\sgn(K_s),$$
where $K_s$ is the Gauss curvature of the immersion at $s$. On the other hand, when 
we are passing  zeros of the polynomial $K_s(t)=\det(A_s+tP_s)$, 
the sign $\sgn(K_s)$ may change.  As it will follow from the computations below - 
see Section~\ref{section:proof of the main results}, we may choose
an orientating frame on $T_sS$, such that the operator  $P_s$ is represented by the matrix
$$0\oplus\pm
\( 
\underbrace{\(\begin{array}{cc}0&-1\\1&0 \end{array}\)
\oplus 
\ldots
\oplus  
\(\begin{array}{cc}0&-1\\1&0 \end{array}\)}_{n-1}
\).
$$
Let $(a_{ij})$ be a symmetric $(2n-1)\times(2n-1)$ matrix representing  
the shape operator $A_s$ in the above frame. Then, for example, in the case $n=2$
we have that 
$$
\det(A_s+tP_s)=\det
\(
\begin{array}{ccc}a_{11}&a_{12}&a_{13}\\a_{12}&a_{22}&a_{23}-t\\a_{13}&a_{23}+t&a_{33} \end{array}
\)=a_{11}t^2 + K_s.
$$

\subsection{The statement of the main result.}
Let $\mathcal{N}:=\mathcal{N}(p, \Sigma, \omega_0)$ be the algebraic number
$$\mathcal{N}:=
\sum\limits_{\ell\in\mathcal{L}}\varepsilon_\ell$$ of lines 
in $\RR^{2n}$ passing through $p$ and tangent to 
$\Sigma$ parallel to the 1-dimensional characteristic distribution 
$\ker\left(\omega_0\big|_{T\Sigma}\right)\subset T\Sigma$ of $\omega_0$. 
Denote $\ind(\Sigma):=\deg(G)$ -- the degree of the Gauss map 
(higher dimensional Whitney index). Denote also
$\ind_p(\Sigma):= \deg(\xi)$ -- the degree of the map 
$\xi:S\to\mathbb{S}^{2n-1}$ given by
$\xi(s)=\xi_{\imath(s)}:=\cfrac{\imath(s)-p}{\|\imath(s)-p\|}$ 
(higher dimensional index of $p$ w.r.t. $\Sigma$). 
It can  be interpreted as the linking number of a $0$-chain
$\{\infty\}-\{p\}$ with $\Sigma$ in $\RR^{2n}$, where $\{\infty\}$ is a
generic point ''near infinity" of $\RR^{2n}$. Note that the linking number is the intersection
number 
of a $1$-chain $[p,\infty]$ with $\Sigma$ in $\RR^{2n}$. 
The main result of  this work is the following

\begin{thm}\label{thm:main result}
Let $(p,\Sigma)$ be in general position as in
Section~\ref{subsect:generalpositionand signs}. Then
\begin{equation}\label{eqn:main formula}
\mathcal{N}=2\ind(\Sigma)-2\ind_p(\Sigma) 
\end{equation}
In particular, the number $\mathcal{N}$  is invariant 
under local regular homotopies  of the pair $(p, \Sigma)$ in the class of
general
position. 
\end{thm}
Note that  the number $\mathcal{N}$ provides a low bound to the actual number of counted
isotropic lines. 

\section{The proof of the main results.}\label{section:proof of the main results}
Consider a smooth manifold  $M:=\SSS^{2n-1}\times\SSS^{2n-1}$. Let 
$X_\pm:=\phi_\pm(S)$ be immersed orientable submanifolds of
$M$ of dimension $2n-1$,  where  immersions
$\phi_\pm:S\to M$
are given by
$$\phi_\pm(s):=(\xi, \pm J_0\circ G)(s)=\left(\xi_q:=\cfrac{q-p}{\|q-p\|}, \pm
J_0(N(q))\right),\ q=\imath(s).$$
Denote $X:=X_+\cup X_-$ and note that 
$X_+\cap X_-=\varnothing$.  In addition, let
$\mathcal{D}:=\Delta(\SSS^{2n-1})$, where $\Delta:\SSS^{2n-1}\to M$
is  the diagonal embedding  $\Delta(q)=(q, q)$.
Every point $x\in X\cap\mathcal{D}$ corresponds bijectively to 
some line $\ell(x)\in\mathcal{L}$. Since the pair $(p, \Sigma)$ is in general 
position, we have that $X$ and $\mathcal{D}$ intersect transversally 
in finitely many points. Consider the intersection number $I(\mathcal{D}, X;M)$
of $\mathcal{D}$ with $X$ in $M$
$$I(\mathcal{D}, X;M)=\sum_{x\in X\pitchfork\mathcal{D}}I_x(\mathcal{D}, X;M),$$
where $I_x(\mathcal{D}, X;M)$ is the local intersection number.

\begin{prop}
For every $x\in X\cap\mathcal{D}$, we have $I_x(\mathcal{D},
X;M)=\varepsilon_{\ell(x)}$.
\end{prop}
\begin{proof}
Suppose that $x=(\xi_q, \pm J_0N(q))\in X\pitchfork\mathcal{D}$, 
in particular, $\xi_q = \pm J_0N(q)$. Let us find the local
intersection number $I_x(\mathcal{D}, X;M)$  of $\mathcal{D}$ with $X$
in $M$ at the point $x$. We fix an orientation
$o_{M}:=o_{\sss^{2n-1}}\times
o_{\sss^{2n-1}}$ on $M$, where $o_{\sss^{2n-1}}$ is the
orientation on $\SSS^{2n-1}$ defined by the outer normal
vector field $\nu$, so that $\nu\times o_{\sss^{2n-1}} =
o_{\rr^{2n}}$. Note that
$T_q(\RR^{2n}\minus\{p\})=\RR\xi_q\oplus(\RR\xi_q)^{\perp}$,
and $(\RR\xi_q)^{\perp}=T_{\xi_q}\SSS^{2n-1}$. So one can
choose mutually  orthogonal unit vectors $e_1, \ldots,
e_{n-1}\in T_{\xi_q}\SSS^{2n-1}$, such that
$$T_q(\RR^{2n}\minus\{p\})=\Span_{\rr}(\xi_q, J_0\xi_q, e_1,
J_0e_1, \ldots, e_{n-1}, J_0e_{n-1})$$ and the orientation
of the ordered basis $(\xi_q, J_0\xi_q, e_1, J_0e_1, \ldots,
e_{n-1}, J_0e_{n-1})$ equals $o_{\rr^{2n}}$. Note that in
this basis we also have that
$$T_{\xi_q}\SSS^{2n-1}=\Span_{\rr}(J_0\xi_q, e_1, J_0e_1,
\ldots, e_{n-1}, J_0e_{n-1})$$ and the orientation of the
ordered basis $(J_0\xi_q, e_1, J_0e_1, \ldots, e_{n-1},
J_0e_{n-1})$ equals $o_{\sss^{2n-1}}$. Next, since
$T_q\Sigma=(\RR N(q))^{\perp}=(\RR(- J_0\xi_q))^{\perp}$ we
have that $$T_q\Sigma=\Span_{\rr}(\pm\xi_q, e_1, J_0e_1,
\ldots, e_{n-1}, J_0e_{n-1})$$ and the orientation of the
ordered basis $(\pm\xi_q, e_1, J_0e_1, \ldots, e_{n-1},
J_0e_{n-1})$ equals  $o_{\Sigma}$.
Now, if $s=\imath^{-1}(q)\in S$, the local intersection number $I_x(\mathcal{D},
X;M)$ equals to
the sign of $\det(d_{\xi_q}\Delta+d_s\phi_\pm)$ (depending on $x\in X_+$ 
or $x\in X_-$) in the following frames:
$$
 \Span_{\rr}\(J_0\xi_q, (e_i, J_0e_i)_{i=1}^{n-1}\)
\oplus
(d_s\imath)^{-1}\(\Span_{\rr}\(\pm\xi_q, (e_i, J_0e_i)_{i=1}^{n-1}\)\)
$$
for $T_{\xi_q}\SSS^{2n-1}\times T_sS$ and 
$$
\Span_{\rr}\(J_0\xi_q, (e_i, J_0e_i)_{i=1}^{n-1}\)\oplus\Span_{\rr}\(J_0\xi_q,
(e_i, J_0e_i)_{i=1}^{n-1}\).
$$
for $T_{(\xi_q, \xi_q)}M$. In these frames the matrix $A_\pm$ of 
$d_{\xi_q}\Delta+d_s\phi_\pm$ equals to
\begin{equation}\label{equation:Jacobian matrix of intersection}
A_\pm=
\(
\begin{array}{cc}
\Id_{2n-1}&\lambda_sDiag(0,1,\ldots,1)
\\
\\
\Id_{2n-1}&B_\pm
\end{array}
\),
\end{equation}
where $\lambda_s=\|q-p\|^{-1}$ and $B_\pm$ is the matrix 
of the differential 
 $d_s\(\pm J_0\circ G\)=\(\pm J_0|_{T_q\Sigma}\)\circ d_sG$.

It follows that  $\det(A_\pm)=\det\(B_\pm-\lambda_sDiag(0,1,\ldots,1)\)$. 
In the chosen frames the matrix of $\pm J_0|_{T_q\Sigma}$  equals
$$C_\pm:=1
\oplus 
\pm
\( 
\underbrace{\(\begin{array}{cc}0&-1\\1&0 \end{array}\)
\oplus 
\ldots
\oplus  
\(\begin{array}{cc}0&-1\\1&0 \end{array}\)}_{n-1}
\)
$$
\noindent and the matrix of $\pm\lambda_sP_s$ equals 
$\lambda_sC_\pm Diag(0,1,\ldots,1).$
Moreover, $C_+C_-=C_-C_+=\Id_{2n-1}$ and $\det(C_\pm)=1$. As a consequence, 
$$\det(A_\pm)=\det\(d_sG\pm\lambda_sP_s\)=\det\(A_s\pm\lambda_sP_s\)$$
and hence $I_x(\mathcal{D}, X;M)=\varepsilon_{\ell(x)}$.
\end{proof}

\begin{cor}
We have \quad $\displaystyle \mathcal{N}=\sum\limits_{\ell\in\mathcal{L}(p,
\Sigma)}\varepsilon_\ell=I(\mathcal{D},
X;M).$
\end{cor}
The next proposition finishes the proof of the main theorem
\begin{prop}
We have \quad $I(\mathcal{D}, X;M)=2\ind(\Sigma)-2\ind_p(\Sigma)$. 
\end{prop}
\begin{proof}

We shall use the homological interpretation of the intersection number. 
Namely, $$I(\mathcal{D},
X;M)=\Delta_*[\mathbb{S}^{2n-1}]\bullet(\phi_+)_*[S]+\Delta_*[\mathbb{S}^{2n-1}]
\bullet (\phi_-)_*[S],$$
where 
$$\bullet:H_{2n-1}(M;\mathbb{Z})\times
H_{2n-1}(M;\mathbb{Z})\to H_0(M;\mathbb{Z})\cong\mathbb{Z}$$
is the homological intersection product, i.e.
$$\alpha\bullet\beta=\(PD(\alpha)\cup PD(\beta)\)\cap[M] = PD(\beta)\cap\alpha =
- PD(\alpha)\cap\beta.$$
Note that for any $\theta_1, \theta_2\in\mathbb{S}^{2n-1}$, the class
$\Delta_*[\mathbb{S}^{2n-1}]$ splits as 
$$\Delta_*[\mathbb{S}^{2n-1}]=\(\Delta_{\theta_1}\)_*[\mathbb{S}^{2n-1}]
+\(\Delta_{\theta_2}\)_*[\mathbb{S}^{2n-1}],$$
where the embeddings $\Delta_{\theta_1}, \Delta_{\theta_2}:\mathbb{S}^{2n-1}\to
M$ are given by
$$\Delta_{\theta_1}(\theta)=(\theta_1,
\theta)\;\;\text{and}\;\;\Delta_{\theta_2}(\theta)=(\theta, \theta_2).$$
Next we fix $\theta_1, \theta_2\in\mathbb{S}^{2n-1}$, such that $\theta_1$ 
is a regular value of the first coordinate 
$\xi:S\to\mathbb{S}^{2n-1}$ of the map $\phi_\pm$, $\theta_2$  
is a regular value of the second coordinate 
$\pm J_0\circ G:S\to\mathbb{S}^{2n-1}$ of the map $\phi_\pm$ 
and  a smooth submanifold
$\mathcal{D}_{\theta_i}:=\Delta_{\theta_i}(\mathbb{S}^{2n-1}),\ 
i=1,2$ intersects transversally with $X$ in nonsingular points . 
Since the manifold $\mathcal{D}_{\theta_i}$ represents the class 
$\(\Delta_{\theta_i}\)_*[\mathbb{S}^{2n-1}]$ for $i=1,2$, it follows
$$I(\mathcal{D}, X;M)=I(\mathcal{D}_{\theta_1}, X;M)+I(\mathcal{D}_{\theta_2},
X;M).$$
It remains to prove the following 
\begin{lem}\label{lemma: intersection equals index}
We have $I(\mathcal{D}_{\theta_1}, X;M)=-2\ind_p(\Sigma)$ and 
$I(\mathcal{D}_{\theta_2}, X;M)=2\ind(\Sigma)$.
\end{lem}

\noindent\textit{Proof of Lemma~\ref{lemma: intersection equals index}.}\; Firstly, we treat $$I(\mathcal{D}_{\theta_1}, X;M)=
I(\mathcal{D}_{\theta_1}, X_+;M)+I(\mathcal{D}_{\theta_1}, X_-;M).$$
We have 
$$I(\mathcal{D}_{\theta_1}, X_\pm;M)=
\sum_{x_\pm\in\mathcal{D}_{\theta_1}\pitchfork
X_\pm}I_{x_\pm}(\mathcal{D}_{\theta_1}, X_\pm;M)$$
and in the chosen frames
$$
I_{x_\pm}(\mathcal{D}_{\theta_1}, X_\pm;M)=\sgn\det
\(
\begin{array}{cc}
0_{2n-1}&d_{s(x_\pm)}\xi
\\
\\
\Id_{2n-1}&d_{s(x_\pm)}\(\pm J_0\circ G\)
\end{array}
\),
$$
where ${s(x_\pm)}=\phi_\pm^{-1}(x_\pm)$. It follows that 
$$I(\mathcal{D}_{\theta_1}, X_\pm;M)=
\sum_{x_\pm\in\mathcal{D}_{\theta_1}\pitchfork
X_\pm}-\sgn\det\(d_{s(x_\pm)}\xi\)=-\deg(\xi)$$
and hence, $I(\mathcal{D}_{\theta_1}, X;M)=-2\ind_p(\Sigma)$.

In the same way we get that 
$$I(\mathcal{D}_{\theta_2}, X_\pm;M)=
\sum_{x_\pm\in\mathcal{D}_{\theta_2}\pitchfork
X_\pm}I_{x_\pm}(\mathcal{D}_{\theta_2}, X_\pm;M)$$
and in the chosen frames
$$
I_{x_\pm}(\mathcal{D}_{\theta_2}, X_\pm;M)=\sgn\det
\(
\begin{array}{cc}
\Id_{2n-1}&d_{s(x_\pm)}\xi
\\
\\
0_{2n-1}&d_{s(x_\pm)}\(\pm J_0\circ G\)
\end{array}
\),
$$
where ${s(x_\pm)}=\phi_\pm^{-1}(x_\pm)$. Recall that $\det(\pm
J_0|_{T_{\imath(s(x_\pm))}\Sigma})=1$. 
It follows that 
$$I(\mathcal{D}_{\theta_2}, X_\pm;M)=
\sum_{x_\pm\in\mathcal{D}_{\theta_2}\pitchfork X_\pm}\sgn\det\(d_{s(x_\pm)}\(\pm
J_0\circ G\)\)=\deg(G)$$
and hence, $I(\mathcal{D}_{\theta_2}, X;M)=2\ind(\Sigma)$.
\end{proof}

Using the Poincar\'{e} duality, we immediately get an integral formula for $\mathcal{N}$.
Indeed, let
$\delta\in H_{dR}^{2n-1}(M;\mathbb{R})$ be the Poincar\'{e}  dual class of
$\Delta_*[\mathbb{S}^{2n-1}]$
in the de Rham cohomology. Explicitly,  the class $\delta$ is given by
$[pr_1^*\mu]+[pr_2^*\mu]$,
where $\mu\in\Omega^{2n-1}(\SSS^{2n-1};\RR)$ is the volume form normalized 
by  $\int\limits_{\sss^{2n-1}}\mu=1$ and $pr_1, pr_2 :\SSS^{2n-1}\times
\SSS^{2n-1}\to\SSS^{2n-1}$  
are the natural projections. 
\begin{cor}
We have \quad  $\displaystyle
\mathcal{N}=-\int\limits_S(\phi_+^*(\delta)+\phi_-^*(\delta))$.
\end{cor}

\end{document}